\documentclass[12pt,oneside]{article}
\usepackage{setspace}
\usepackage[margin=1in]{geometry}
\usepackage[colorlinks=true,citecolor=blue]{hyperref}
\usepackage[mathscr]{euscript}
\usepackage[utf8]{inputenc}
\usepackage[margin=.5in]{caption}

\usepackage{amsmath}%
\usepackage{amsfonts}%
\usepackage{amsthm}
\usepackage{amssymb}%
\usepackage{graphicx}
\usepackage{units}
\usepackage{verbatim}
\usepackage{enumerate}
\usepackage{tikz}
\usepackage{mathtools}
\usepackage{needspace}
\usepackage[inline]{enumitem}

\usepackage{parskip}

\begingroup
    \makeatletter
    \@for\theoremstyle:=definition,remark,plain\do{%
        \expandafter\g@addto@macro\csname th@\theoremstyle\endcsname{%
            \addtolength\thm@preskip\parskip
            }%
        }
\endgroup

\newtheorem{theorem}{Theorem}[section]
\newtheorem{lemma}[theorem]{Lemma}

\theoremstyle{remark}

\newcommand{\mb}{\mathbb}
\newcommand{\mr}{\mathrm}
\newcommand{\df}[1]{\textbf{\textit{\color{cyan!10!black} #1}}}
\newcommand{\eps}{\varepsilon}
\renewcommand{\phi}{\varphi}
\DeclareMathOperator{\dist}{dist}

\newcommand{\sphere}{\mathbf{S}}

\DeclareMathOperator{\proj}{proj}

\DeclareMathOperator{\graph}{graph}

\DeclareMathOperator{\kgn}{k_gn}
\DeclareMathOperator{\lsf}{lsf}

\newcommand{\jord}{\mathfrak{J}}

\newcommand{\jn}{\mathfrak{J}_\mathrm{N}}
\newcommand{\jci}{\mathfrak{J}_\mathrm{C^\infty}}

\newcommand{\bis}{\mathfrak{B}}

\DeclareMathOperator{\cf}{cf}

\begin{document}
%\raggedright

\begin{center}
{\Large \noindent
Continuous dependence of curvature flow on initial conditions 
%for curves bisecting the 2-sphere
}

\bigskip

%\begin{spacing}{1.5}
{\large \noindent
Michael~Gene~Dobbins
%\textsuperscript{1}
}
%\end{spacing}

%\smallskip

\begin{minipage}{0.85\textwidth}
\raggedright \footnotesize \singlespacing
\noindent
%\llap{\textsuperscript{1}}
Department of Mathematical Sciences, Binghamton University (SUNY), Binghamton, \newline New York, USA. 
\texttt{mdobbins@binghamton.edu} 
\end{minipage}

\end{center}

\bigskip

\begin{abstract}
We study the evolution of a Jordan curve on the 2-sphere by curvature flow, also known as curve shortening flow, and by level-set flow, which is a weak formulation of curvature flow. 
We show that the evolution of the curve depends continuously on the initial curve in Fréchet distance in the case where the curve bisects the sphere.  This even holds in the limit as time goes to infinity.  
This builds on Joseph Lauer's work on existence and uniqueness of solutions to the curvature flow problem on the sphere when the initial curve is not smooth. 
\end{abstract}

%\begin{center}
%Math Subject Classification (2010) 
%\end{center}

\section{Introduction}
%Let us briefly describe the main result before getting to precise definitions. 
Consider the space $\bis$ of all simple closed curves, called bisectors, in the 2-sphere such that each region on either side has area $2\pi$, with Fréchet distance the metric on $\bis$.
%where the distance between a pair of curves is the infimum of sup-distance over all pairs of parameterizations of the the curves. 
Joseph Lauer recently showed existence and uniqueness of solution to the curvature flow problem for all positive time when the initial curve is a bisector \cite{lauer2016evolution}.  
Earlier, Michael Gage showed that a bisector evolving by curvature flow approaches a great circle in the limit as time goes to infinity \cite{gage1990curve}.   
This defines a function 
$ \cf : \bis \times [0,\infty]_\mb{R} \to \bis $  
by evolution by curvature flow.  
The main result of this paper is that this function is continuous.

\begin{theorem}\label{theoremFlowConvergesInFrechet}
Let $\gamma_k \in \bis$ and $t_k \in [0,\infty]_\mb{R}$ for $k \in \{1,\dots,\infty\}$.
If $\gamma_k \to \gamma_\infty$ in Fréchet distance and $t_k \to t_\infty$, then $\cf(\gamma_k,t_k) \to \cf(\gamma_\infty,t_\infty)$ in Fréchet distance. 
\end{theorem}

\subsection{Motivation from oriented matroids and vector bundles}

Whether the solution to a differential equation depends continuously on initial data is a question of fundamental interest.  In this case, however, the author was motivated by a conjecture from combinatorics, which in turn has applications for working with vector bundles.  
Specifically, Nicolai Mnëv and Günter Ziegler conjectured that each OM-Grassmannian of a realizable oriented matroid is homotopy equivalent to the corresponding real Grassmannian \cite[Conjecture 2.2]{mnev1993combinatorial}.  We will not go into precise technical details of the conjecture, 
%since it is not needed for the proof of Theorem \ref{theoremFlowConvergesInFrechet}, 
but we will see an informal picture. 
The material in this section is not needed to understand the rest of the paper. 

An oriented matroid is a combinatorial analog to a real vector space, but where we only keep track of sign information.  One way to obtain an oriented matroid is as the set of all sequences of signs of all the vectors in a vector subspace of $\mb{R}^n$.  For example, the oriented matroid obtained from the space $\{(x,y) \in \mb{R}^2 : x+y = 0\}$ is given by the set  
$\{(+,-),(0,0),(-,+)\}$. 
However, oriented matroids are defined by purely combinatorial axioms, and not all oriented matroids are obtained in this way.  One way the Grassmannian is defined is as the set of all $k$-dimensional vector subspaces of a given vector space, and this set is given a metric.  Oriented matroids come equipped with analogs of dimension and subspace, and an OM-Grassmannian is a finite simplicial complex that is defined analogously from an oriented matroid. 

Gaku Liu showed that this conjecture does not hold in general; i.e.\ there is an OM-Grassmannian of a realizable oriented matroid that is not homotopy equivalent to the corresponding real Grassmannian \cite{liu2020counterexample}.  
Although the full conjecture is false, special cases remain open, such as for MacPhersonians, which are the OM-Grassmannians of the oriented matroids corresponding to $\mb{R}^n$, and in particular for the rank 3 case, which is the analog of the Grassmannian consisting of 3-dimensional subspaces of $\mb{R}^n$.
Mnëv and Ziegler had listed the rank 3 case as a theorem with the expectation that a proof would later appear in Eric Babson's Ph.D.\ thesis, but then it did not \cite{mnev1993combinatorial,babson1993combinatorial}.  Also, an erroneous proof that each MacPhersonian is homotopy equivalent to the corresponding Grassmannian in all ranks was published and then retracted \cite{biss2003homotopy,biss2009erratum}.

One source of interest in Grassmannians is as classifying spaces for vector bundles.  If the conjectured homotopy equivalence with the MacPhersonian were true, then it would mean that matroid bundles, which are purely combinatorial objects, would be an effective way to represent vector bundles, and matroid bundles could then serve as the basis for developing data structures and algorithms for working with vector bundles. 

Jim Lawrence showed that oriented matroids can be characterized as the combinatorial cell decompositions of the sphere arising from essential pseudosphere arrangements \cite{folkman1978oriented}.  Lawrence's theorem provides a topological model for the combinatorial axioms of oriented matroids.  In the rank 3 case, these are pseudocircle arrangements, which are collections of oriented simple closed curves in the 2-sphere such that every pair of curves either coincide or intersect at exactly 2 points, in which case any third curve either separates the 2 points or passes though both points.  An arrangement is said to be essential when no single point is contained in all pseudospheres. 

Dobbins introduced spaces of weighted essential pseudosphere arrangements called pseudolinear Grassmannians, and showed that each rank 3 pseudolinear Grassmannian is homotopy equivalent to the corresponding real Grassmannian.  
The pseudolinear Grassmannians serve as an intermediate space between the Grassmannians and MacPhersonians, and Dobbins's theorem represents a step toward showing that the conjecture holds for the rank 3 MacPhersonians by showing homotopy equivalence between this intermediate space and one side \cite{dobbins2017grassmannians}.  

Thw present paper takes another step toward showing the rank 3 MacPhersonian case of the conjecture, by providing a tool for replacing pseudolinear Grassmannians with simpler spaces, namely spaces of \emph{antipodally   symmetric} weighted pseudocircle arrangements.  These spaces have the advantage that we can disregard the last property of pseudocircle arrangements, i.e.\ a third curve must separates or pass though the points where two other curves intersect, since this property is already guaranteed to hold in the symmetric case.

Most of the proof of Dobbins's theorem is occupied with showing that a pseudocircle arrangement can be deformed into a great circle arrangement in a way that depends continuously on the initial arrangement, and that satisfies the conditions of being a pseudocircle arrangement throughout the deformation.
Armed with Theorem \ref{theoremFlowConvergesInFrechet}, we might be tempted to define such a deformation by simply allowing all pseudocircles to evolve by curvature flow to great circles.
A problem with this is that arrangements in the pseudolinear Grassmannian must be essential, i.e.\ no point can be in the intersection of all pseudocircles.  However, it may happen that an arrangement of bisectors that was initially essential could all pass though a single point at some later time during the course of their evolution.  Instead, Theorem \ref{theoremFlowConvergesInFrechet} just gives us a way to deform a single pseudocircle to a great circle in an antipodally symmetric way, and doing so for an arrangement will take a similar argument to the proof of Dobbins's theorem.  Such an argument would be beyond the scope of the present paper.

\subsection{Definitions and notation}

%We will generally use square brackets around functions and round brackets around arguments passed to functions.
%We use this for partial function application. 
%For example, given a function $f : W \times X \to Y$ and $w \in W$, 
%we have $f(w) : X \to Y$ by $[f(w)](x) = f(w,x)$. 
%We will also use square brackets when composing and inverting functions.
%For example, 
%given $f : W \times X \to Y$ and $g : Z \to Y$, we write $x = [[f(w)]^{-1}\circ g](z)$ for the value $x$ such that $f(w,x) = g(z)$, provided that $f(w) : X \to Y$ is invertible. 
%We distinguish multiplicative inverse by using round brackets when it is helpful, i.e.\  $(x)^{-1} = \nicefrac1x$. 

We denote real intervals by $(a,b]_\mb{R} = \{x \in \mb{R} : a < x \leq b\}$ for bounds $a,b \in \mb{R}$ with any combination of round or square brackets for (half) open or closed intervals. 
We give infinite closed intervals the same topology as finite closed intervals.
For instance, $[0,\infty]_\mb{R}$ is homeomorphic to $[0,1]_\mb{R}$ and with this topology the sequence of natural numbers converge to $\infty$. 
Be warned that that the topology on $[0,\infty]_\mb{R}$ is not the same as the usual metric topology. 
We also similarly denote segments in the complex plane by 
$[a,b]_\mb{C} = \{ ta +(1-t)b : t \in [0,1]_\mb{R} \}$ for $a,b \in \mb{C}$. 
We denote the unit circle in the complex plane by $\sphere^1$ and the unit sphere in $\mb{R}^3$ by $\sphere^2$.
For $X \subseteq \sphere^2$ and $\delta > 0$, let $X \oplus \delta$ be the set of points at most distance $\delta$ from $X$.  

Let $\jord$ denote the set of simple closed curves, called Jordan curves, on the 2-sphere, which we treat as subsets of the 2-sphere.
\df{Fréchet distance} on $\jord$ is defined by 
\[ \dist_\mr{F} (\gamma_1,\gamma_0) = \inf_{\phi_1,\phi_0} \sup_x \|\phi_1(x) -\phi_0(x) \|  \]
where the $\phi_i : \sphere^1 \to \gamma_i$ are a homeomorphisms. 
We will also use Hausdorff distance, which defines a strictly coarser metric on $\jord$ than Fréchet distance does.
Hausdorff distance is defined by 
\[
\dist_H(\gamma_1,\gamma_0) = \inf \left\{ \delta : \gamma_1 \subseteq \gamma_0 \oplus \delta, \gamma_0 \subseteq \gamma_1 \oplus \delta \right\}. 
\]

Let $\jci$ be the subset of smooth curves and $\jn$ be the subset of curves of area 0. 
Note that $\jn$ is known to be a proper subset of $\jord$, as there are Jordan curves that have positive area \cite{osgood1903jordan}. 
Given a smooth curve $\gamma \in \jci$ and a point $p \in \gamma$, let $\kgn(\gamma,p) \in \mb{R}^3$ denote the vector of geodesic curvature of $\gamma$ at $p$.
%, i.e. $ \mr{k_gn}(\gamma,p) = \proj_{p^\bot}([\mr{D}^2 \phi](x)) $
%where $\mr{D}^2\phi$ is the second derivative of a unit speed parameterization $\phi : \mb{R} \to \gamma$ and $\phi(x) = p$, and $\proj_{p^\bot}$ is orthogonal projection to the plane perpendicular to $p$.   
A solution to the \df{curvature flow problem} for a 
given initial curve $\gamma_0 \in \jord$ and stopping time $T \in (0,\infty]_\mb{R}$, 
is a map $\Gamma : \sphere^1 \times (0,T) \to \sphere^2$ such that $\partial_t \Gamma(x,t) = \kgn(\Gamma(\sphere^1,t),\Gamma(x,t))$ 
and $\Gamma(\sphere^1,t) \to \gamma_0$ in Fréchet distance as $t \to 0$. 
Let 
\[\cf(\gamma_0,t) = \Gamma(\sphere^1,t)\] 
where $\Gamma$ is the solution to the curvature flow problem starting from $\gamma_0$, provided that a unique solution exists. 
Lauer showed $\cf(\gamma_0,t)$ is well defined for $t$ sufficiently small if $\gamma_0 \in \jn$, and for all $t > 0$ if $\gamma_0 \in \bis$ \cite{lauer2016evolution}.
A theorem of Michael Gage implies that for $\gamma_0 \in \bis$, we have $\cf(\gamma_0,t)$ approaches a great circle as $t \to \infty$, which we simply denote by $\cf(\gamma_0, \infty)$ \cite{gage1990curve}.

For $T \in [0,\infty]_\mb{R}$, let 
$\jord_T$ be the set of curves for which the curvature flow problem has a solution up to time $T$. 
%$\jord_T = \{ \gamma \in \jord : \cf(\gamma,T) \in \jord \textrm{ is defined}\}$ 
%be the set of curves such that curvature flow is defined and is a Jordan curve up to time $T$.  
Note that $\jord_\infty = \bis$.

\subsection{Level-set flow}

Lauer's theorem uses a weak notion of curvature flow called level-set flow, which is a set that evolves from a given initial set $X \subseteq \sphere^2$ by 
\[ \lsf(X,t) = 
\sphere^2 \setminus \bigcup \left\{ \cf(\gamma,t-s) :  s<t, \gamma \in \jci, \gamma \cap \lsf(X,s) = \emptyset \right\}. 
\]
In our context, an equivalent simpler definition of level-set flow suffices.
Given a curve $\gamma \in \jord$, let $A_k(0) \subset \sphere^2$ be an intial sequence of annuli with smooth boundary, 
and let $A_k(t)$ be parameterized family of sets with boundary evolving by curvature flow, i.e.\ $\partial A_k(t) = \cf(\alpha_k,t) \cup \cf(\alpha_{-k},t)$ where $\partial A_k(0) = \alpha_k\cup \alpha_{-k}$ with $\alpha_{\pm k} \in \jci$.
Then, 
\[ \lsf(\gamma,t) = \bigcap_{k=1}^\infty A_k(t).\] 
In this case we say the annuli $A_k(t)$ evolve by curvature flow \cite{chen1991uniqueness,evans1991motion,ilmanen1994elliptic,lauer2016evolution}. 

Lauer showed that if the initial curve has area 0, then the level-set flow immediately becomes a smooth Jordan curve evolving by curvature flow up to some time $T> 0$, after which the level-set flow vanishes.
Moreover, the level-set flow is the unique solution to the curvature flow problem in this case. 
That is, if $\gamma \in \jn$ then $\lsf(\gamma,t) \in \jci$ for $t>0$, and also $\cf(\gamma,t) = \lsf(\gamma,t)$ \cite{lauer2016evolution}.
Note that this process is not time dependent, i.e.\ 
$\cf(\cf(\gamma,s),t) = \cf(\gamma,s+t)$. 
%and 
%$\lsf(\lsf(\gamma,s),t) = \lsf(\gamma,s+t)$.

%for a curve $\gamma \in \jn$ with area 0, it holds that $\lsf(\gamma,t) \in \jci$ is smooth for all $t \in (0,T)_\mb{R}$ up to some time $T> 0$, after which the level-set flow vanishes.
%Also, for a smooth curve $\gamma \in \jci$, level-set flow coincides with curvature flow, i.e.\ . 

A nice feature of curvature flow, which provides the underlying intuition behind the definition for level-set flow, is the avoidance principle, which says that if two curves are initially disjoint then they remain disjoint throughout their evolution, for as long as a solution exists. 
Sigurd Angenet showed something more specific, that the number of intersection points immediately becomes finite and then is non-increasing throughout their evolution, provided only that the two curves are initially distinct \cite[Theorem~1.3]{angenent1991parabolic}.

\section{Proof of continuity}

%The argument below depends on the curve evenly bisecting the sphere in Footnote \ref{footnote-gammaevenlybisects}.

The goal of this section is to prove Theorem \ref{theoremFlowConvergesInFrechet}.
Let us first briefly outline the main idea of the proof. 
We will argue by contradiction, starting from the assumption that the hypotheses of Theorem \ref{theoremFlowConvergesInFrechet} hold, but that $\cf(\gamma_k,t_k)$ does not converge to $\cf(\gamma_\infty,t_\infty)$ in Fréchet distance.
In Lemma \ref{lemma-flowconvergesinhausdorff}, we show that $\cf(\gamma_k,t_k) \to \cf(\gamma_\infty,t_\infty)$ in Hausdorff distance, which follows almost directly from the definition of level-set flow; see Figure \ref{figure-Hausdorff-converge}.  
In order for $\cf(\gamma_k,t_k)$ to converge in Hausdorff distance but not Fréchet distance, 
the curve $\cf(\gamma_k,t_k)$ must pass close by some point $x_\infty \in \cf(\gamma_\infty,t_\infty)$ at least three times as it moves back and forth alongside $\cf(\gamma_\infty,t_\infty)$; see Figure \ref{figure-Frechet-converge}. 
Using Lemma \ref{lemma-crossingcurve}, 
we will find a curve $\eta$ that %intersects $\gamma_\infty$ at only 2 points, and also 
intersects $\gamma_k$ for some $k$ sufficiently large at only 2 points, and evenly bisects the sphere, and 
passes close by $x_\infty$ at time $t_k$.
%such that $\cf(\eta,t_\infty)$ intersects $\cf(\gamma_\infty,t_\infty)$ at .    
%$\zeta$ slightly so that it intersects both $\gamma_\infty$ and  at exactly 2 points each.
Then, $\cf(\gamma_k,t_k)$ will have to cross $\cf(\eta,t_k)$ at least 3 times as it passes by $x_\infty$ while moving back and forth as alongside $\cf(\gamma_\infty,t_\infty)$, which contradicts that the number intersection points is non-increasing, since $\gamma_k$ only intersects $\eta$ twice; see Figure \ref{figure-Frechet-converge}. 
Some parts of the argument will have to be modified in the infinite time case where $t_\infty = \infty$, so we will deal with that case separately at the end.

Let us begin with a basic lemma. 

\vbox{
\begin{lemma}\label{lemma-nearparam}
For each simple closed curve $\gamma \in \jord$ and $\eps>0$, there is $\delta = \delta(\gamma,\eps) >0$ such that for every simple closed curve $\eta \subset (\gamma \oplus \delta)$, there is continuous map a $\phi : \eta \to \gamma$ such that for all $x \in \eta$, $\|\phi(x)-x\| < \eps$.
\end{lemma}
}

Note that $\phi$ is not necessarily bijective, so this does not provide an upper bound on Fréchet distance. 

\begin{proof}[Proof of Lemma \ref{lemma-nearparam}]
Let $p_1$ be any point on $\gamma$, then direct $\gamma$ and let $p_1,p_2,p_3,\dots$ be the sequence of points such that $p_{i+1}$ is the first point of $\gamma$ after $p_i$ that is distance $\eps/6$ from $p_i$, unless the rest of $\gamma$ from $p_i$ to $p_1$ is within distance $\eps/6$, in which case the sequence ends.

Observe that this sequence must be finite.  To see why, suppose that the sequence of points $p_i$ were not finite and map $\gamma$ homeomorphically to a circle.  Then the images of the points $p_i$ would then have an accumulation point, so by continuity the sequence of points $p_i$ would also have an accumulation point, which is a contradiction as each consecutive pair is the same distance apart.  Let $p_n$ be the last point of the sequence, and let $p_{n+i}=p_i$.

Let $\delta$ be $\nicefrac16$ times the minimum distance between non-consecutive pairs of arcs of $\gamma$ subdivided by the points $\{p_i\}$, and let $\eta$ be a simple closed curve in $\gamma\oplus \delta$.
Note that $0<\delta\leq\eps/6$. 
Let $q_1,\dots,q_m$ be an analogously defined sequence on $\eta$ of points spaced $\delta$ apart.  
For each $q_j$, choose some point $\phi(q_j) \in \gamma$ at most $\delta$ away.
Then, for each consecutive pair of point, their images $(\phi(q_j),\phi(q_{j+1}))$ are distance at most $3\delta$ apart, so the points $\phi(q_j),\phi(q_{j+1})$ must either be in the same arc or in consecutive arcs of $\gamma$ subdivided by $\{p_i\}$.  Hence, there is an arc on $\gamma$ from $\phi(q_j)$ to $\phi(q_{j+1})$ that crosses at most one point $p_i$, so this arc has diameter at most $3\eps/6 = \eps/2$.
Therefore, every point on the arc of $\eta$ from $q_j$ to $q_{j+1}$ is at most distance $\eps/2+2\delta<\eps$ from every point on the arc of $\gamma$ from $\phi(q_j)$ to $\phi(q_{j+1})$, so we can extend $\phi$ to a map with the desired properties.
\end{proof}

\begin{lemma}\label{lemma-flowconvergesinhausdorff}
Let $T \in (0,\infty]_\mb{R}$ and $\gamma_k \in \jord_T$ and $t_k \in [0,T)_\mb{R}$ for $k \in \{1,\dots,\infty\}$.
If
%\ $\cf(\gamma_k,t_k) \in \jci$ is defined and 
$\gamma_k \to \gamma_\infty$ in Fréchet distance and $t_k \to t_\infty$, then $\cf(\gamma_k,t_k) \to \cf(\gamma_\infty,t_\infty)$ in Hausdorff distance.
\end{lemma}

%Does this hold for $t_\infty = \infty$? 

\begin{proof}%[Proof of Lemma \ref{lemma-flowconvergesinhausdorff}]
%The case where $t_\infty = 0$ is trivial, so let us assume that $t_\infty > 0$.

By Lauer's theorem, there is some nested sequence of smooth closed annuli $A_{1,0} \supset A^\circ_{1,0} \supset A_{2,0} \supset A^\circ_{2,0} \supset A_{3,0} \supset  \cdots$ such that $\cf(\gamma_\infty,t) = \bigcap_{i=1}^\infty A_i(t)$ %$ = \bigcap_{i=1}^\infty A_i(t)$ 
where the boundary of $A_i(t)$ %is a pair of curves %$\alpha_i(t),\beta_i(t)$ 
evolves by curvature flow starting from %the curves %$\alpha_i(0) = \alpha_{i,0}$ and $\beta_i(0) = \beta_{i,0}$ 
$A_i(0) = A_{i,0}$  \cite[Theorem 1.1]{lauer2016evolution}. 
Note that this only holds for $t< \infty$, since all annuli evolving by curvature flow either vanish or envelop the sphere in the limit as $t \to \infty$, which is why we require $t_\infty$ to be finite. 
We may assume that $T$ is finite, since $t_k$ must be bounded by $2t_\infty$ for all $k$ sufficiently large.

%and for a curve $\gamma \in \jord_T$, let $\graph(\gamma) = \graph(\cf(\gamma))$. 

%Let 
%\[ \graph(\gamma_k) = \left\{(t,\cf(\gamma_k,t)) : t \in [0,T]_\mb{R}\right\}, \quad 
%\graph(A_i) = \left\{(t,A_i(t)): t \in [0,T]_\mb{R}\right\}.  \]
%Then, 
%$\graph(\gamma_\infty) = \bigcap_{i=1}^\infty \graph(A_i) $. 

Consider $\eps>0$.
Let $\delta$ be as implied by Lemma \ref{lemma-nearparam} for the curve $\cf(\gamma_\infty,t_\infty)$ and $\eps$. 
%By Lemma \ref{lemma-nearparam}, there is $\delta$ for every simple closed curve $\eta \subset \gamma_\infty(t_\infty) \oplus \delta$, there is 
For a function $\phi$ from $[0,T]_\mb{R}$ to the power set of $\sphere^2$, let 
\[ \graph(\phi) = \{(t,y): t \in [0,T]_\mb{R}, y \in \phi(t) \}. \]

We claim that for all $i$ sufficiently large, $\graph(A_i) \subset \graph(\cf(\gamma_\infty) \oplus \delta/2)$.
Suppose not.  Then, we can restrict to a subsequence of $A_i$ where for each $i$ there is a point $p_i \in \graph(A_i \setminus (\cf(\gamma_\infty)\oplus\delta/2)) $.
Since $[0,T]_\mb{R}\times\sphere^2$ is compact, there is some subsequence $p_i$ that converges to a point $p = (t,q) \in [0,T]_\mb{R}\times\sphere^2$.
Then, $q$ is at least distance $\delta/2$ from $\gamma_\infty(t)$, so $p \not\in \graph (\cf(\gamma_\infty))$, so there is some $j$ such that $p \not\in\graph(A_j)$.
Since the set $\graph(A_j)$ is closed, $p_\infty$ is a positive distance $\delta_1$ from $\graph(A_j)$.
Since the sets $\graph(A_i)$ are nested, $p_\infty$ is at least distance $\delta_1$ from $\graph(A_i)$ for all $i \geq j$, but that contradicts that $p_\infty$ is the limit of a subsequence of the points $p_i$. 
Hence, the claim holds.

Let $j$ be large enough that $\graph(A_j) \subset \graph(\cf(\gamma_\infty)\oplus \delta/2 )$. 
The annulus boundary $\partial A_j(0)$ and $\gamma_\infty$ are closed and disjoint, so $\partial A_j(0)$ is a positive distance $\delta_2$ from $\gamma_\infty$.
For all $k$ sufficiently large, $\gamma_k \subset \gamma_\infty\oplus\delta_2$, so 
$\gamma_k \subset A_j(0)$, so by the avoidance principle, $\cf(\gamma_k,t) \subset A_j(t)$ for all $t\in[0,T]_\mb{R}$ \cite[Theorem 1.3]{angenent1991parabolic}.
In particular, $\cf(\gamma_k,t_k) \subset A_j(t_k)$, so by the above claim $\cf(\gamma_k,t_k) \subset \cf(\gamma_\infty,t_k)\oplus \delta/2$; see Figure \ref{figure-Hausdorff-converge}.

\begin{figure}
\includegraphics{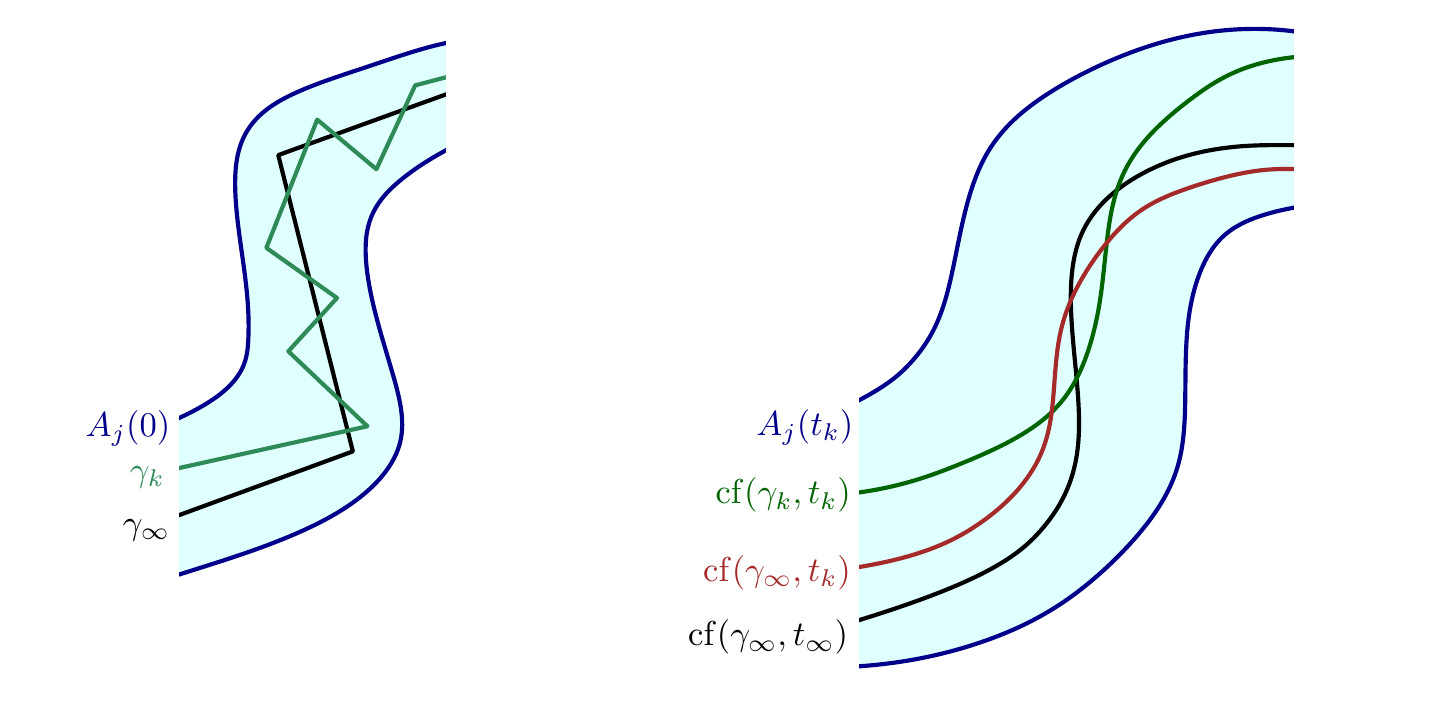}
\caption{To get $\cf(\gamma_k,t_k)$ close to $\cf(\gamma_\infty,t_\infty)$, choose a small annulus $A_j(0)$ about $\gamma_\infty$ to make $A_j(t)$ close to $\cf(\gamma_\infty,t)$ for all $t \leq T$, and make $k$ large enough that $\gamma_k$ is in $A_j(0)$, so  %and $\cf(\gamma_\infty,t_k)$ is close to $\cf(\gamma_\infty,t_\infty)$.
$\cf(\gamma_k,t_k)$ is in $A_j(t_k)$, so $\cf(\gamma_k,t_k)$ is close to $\cf(\gamma_\infty,t_k)$, which is close to $\cf(\gamma_\infty,t_\infty)$.}
\label{figure-Hausdorff-converge}
\end{figure}

By definition, if $t>0$, then $\cf(\gamma_\infty,t)$ is smooth. %\cite{lauer2016evolution}.
In particular, this means that the geodesic curvature of $\cf(\gamma_\infty,t)$ along a unit speed parameterization is continuous, so by the Heine-Cantor theorem, the geodesic curvature of $\cf(\gamma_\infty,t)$ is bounded, which bounds the velocity of each point on $\cf(\gamma_\infty,t)$ moving by curvature flow, which is similarly bounded for all $t$ in a closed interval about $t_\infty$.
Hence, for all $k$ sufficiently large, $\cf(\gamma_\infty,t_k)$ and $\cf(\gamma_\infty,t_\infty)$ are at most $\delta/2$ apart in Fréchet distance.
Thus, $\cf(\gamma_k,t_k) \subset \cf(\gamma_\infty,t_\infty)\oplus\delta$.
By Lemma \ref{lemma-nearparam}, there is a continuous map $\phi: \cf(\gamma_k,t_k) \to \cf(\gamma_\infty,t_\infty)$ such that for all $x \in \cf(\gamma_k,t_k)$, $\|\phi(x)-x\|<\eps$.

Since $T<\infty$, the set $A_i(T)$ is still an annulus for all $i$ sufficiently large.  That is, both curves on the boundary of $A_i$ continue to evolve and are Jordan curves up to time $T$.  Let us just assume that $A_1(T)$ is an annulus. 
Since $\gamma_k \to \gamma_\infty$ in Fréchet distance, $\gamma_k$ winds once around $A_1(0)$ for $k$ sufficiently large, i.e.\ parameterizing $\gamma_k$ gives a loop in a generator of the fundamental group of $A_1(0)$.  
Since $\gamma_k$ cannot intersect the boundary of $A_1$ as they evolve \cite[Theorem~1.3]{angenent1991parabolic}, $\cf(\gamma_k,t)$ winds once around $A_1(t)$ for all $t\in[0,T]_\mb{R}$.
We may choose $\delta_3$ small enough that $A_1^\circ(t) \supset \cf(\gamma_\infty,t_\infty)$ for all $t \in [t_\infty-\delta_3,t_\infty+\delta_3]_\mb{R}$, and $A_1(t) \supset \cf(\gamma_\infty,t_\infty) \oplus \eps$ for all such $t$, provided that $\eps$ is sufficiently small. 
Let $\psi : \cf(\gamma_k,t_k) \times [0,1]_\mb{R} \to \sphere^2$ by $\psi(x,t) = \proj_{\sphere^2}(t\phi(x) +(1-t)x)$ where $\proj_{\sphere^2}$ normalizes non-zero vectors to the unit sphere.
Then, $\psi$ is a homotopy from the identity map on $\cf(\gamma_k,t_k)$ to $\phi$ that is within $\cf(\gamma_\infty,t_\infty) \oplus \eps \subset A_1(tt_\infty+(1-t)t_k)$.
Since $\cf(\gamma_k,t_k)$ winds once around $A_1(t_k)$, 
the map $\phi=\psi(\cdot,1)$ winds once around $A_1(t_\infty)$, which implies that $\phi$ actually winds once around $\cf(\gamma_\infty,t_\infty)$. 
Therefore, $\phi$ is surjective.
Hence, for all $k$ sufficiently large, $\cf(\gamma_\infty,t_\infty) \subset \cf(\gamma_k,t_k)\oplus\eps$ and $\cf(\gamma_k,t_k) \subset \cf(\gamma_\infty,t_\infty)\oplus\eps$,   
%For all $\eps>0$ sufficiently small, this holds for all $k$ sufficiently large, 
so $\dist_H(\cf(\gamma_k,t_k),\cf(\gamma_\infty,t_\infty)) \leq \eps$. 

By letting $\eps \to 0$, we get that 
$\cf(\gamma_k,t_k) \to \cf(\gamma_\infty,t_\infty)$ in Hausdorff distance. 
Note that if $\phi$ were bijective, then we would get convergence in Fréchet distance, but this argument only shows that $\phi$ is surjective. 
\end{proof}

\begin{lemma}\label{lemma-crossingpaths}
Let $\gamma_1, \gamma_2 \in \jord_S$ be a pair of curves that intersect at finitely many points, 
and let $p_\mr{T} \in (\cf(\gamma_1,T) \cap \cf(\gamma_2,T))$ be a point of intersection at time $T < S$.  Then, there is a continuous trajectory $p: [0,T]_\mb{R} \to \sphere^2$ such that for all $t \in [0,T]_\mb{R}$, we have $p(t) \in  (\cf(\gamma_1,t) \cap \cf(\gamma_2,t))$ and $p(T) = p_\mr{T}$.
\end{lemma}
%Furthermore, for each point of intersection $p_0 \in (\gamma_1 \cap \gamma_2)$, there is at most one such trajectory starting from $p(0) = p_0$. 

\begin{proof}
Sigurd Angenet showed that the evolutions by curvature flow starting from distinct curves $\gamma_1$ and $\gamma_2$ meet tangentially at only a discrete set of times \cite[Theorem 1.1]{angenent1991parabolic}, and that the number of intersection points after positive time is finite and non-increasing and remains constant between the times where the curves meet tangentially \cite[Theorem 1.3]{angenent1991parabolic}.  

We can continually track transversal intersection points.
Specifically, given a transversal intersection point $q(t) \in \cf(\gamma_1,t) \cap \cf(\gamma_2,t)$, let $q$ move with velocity 
\[\partial_t q(t) = P_{1,t}(\kgn\cf(\gamma_1,t)) + P_{2,t}(\kgn\cf(\gamma_2,t))\] 
where $P_{1,t}$ is the oblique projection in the tangent fiber over $q(t)$ along the line tangent to $\gamma_1(t)$  to the line tangent to $\gamma_2(t)$, and $P_{2,t}$ is the analogs projection. 
Let us consider the orthogonal projection of $\partial_t q(t)$ to the normal line of $\cf(\gamma_1,t)$. 
Since $P_{1,t}$ projects obliquely along the tangent, the projection to the normal line is unchanged by the action of $P_{1,t}$, so the orthogonal projection of the first term $P_{1,t}(\kgn\cf(\gamma_1,t))$ to the normal line is the vector of geodesic curvature $\kgn\cf(\gamma_2,t)$. 
Since $P_{2,t}$ projects to the tangent, the projection of the second term $P_{2,t}(\kgn\cf(\gamma_2,t))$ to the normal line vanishes.  Hence, the orthogonal projection of $\partial_t q(t)$ to the normal line of $\cf(\gamma_1,t)$ is the same as the velocity of curvature flow, so $q(t)$ moves along $\cf(\gamma_1,t)$ as the curve evolves.  
Similarly, $q(t)$ moves along $\cf(\gamma_2,t)$ as the curve evolves, so 
$q(t)$ continues to be a point of intersection provided that the intersection remains transversal.
Therefore, between times of tangential intersection, the intersection points can be partitioned into a finite set of disjoint trajectories.

We will now see that at times of tangential intersection, the points of intersection may merge, but cannot jump to or appear at new locations.  
That is, each tangential intersection is the limit of intersection points from below in time.

Assume for the sake of contradiction that $T>0$ and $p_\mr{T}$ is a point where $\cf(\gamma_1,T)$ and $\cf(\gamma_2,T)$ intersect tangentially, and that there is some bound $\eps > 0$ such that for all $t < T$ sufficiently close to $T$,  the set $\cf(\gamma_1,t) \cap \cf(\gamma_2,t)$ is bounded away from $p_\mr{T}$ by at least $\eps$. 
Angenet also showed that for a pair of curves evolving by curvature flow on an oriented Riemannian surface, the number of intersection points is decreasing \cite[Theorem 1.3]{angenent1991parabolic}.  
We will make a contradiction by constructing a new surface around the curves $\cf(\gamma_i,T)$ so that their evolution a short time before $T$ violates Angenet's theorem. 

We construct the surface as follows. 
Since $T > 0$, the curves $\cf(\gamma_i,T)$ are smooth, so we can choose a local coordinate system on the surface of the sphere that makes the $\cf(\gamma_i,T)$ implicit functions in a small patch around $p_\mr{T}$.
Specifically, a small arc of $\cf(\gamma_i,T)$ around $p_\mr{T}$ can be expressed by the latitude where each curve of longitude intersects the arc of $\cf(\gamma_i,T)$.
Let $R_0$ be a closed spherically rectangular neighborhood of $p_\mr{T}$ on $\sphere^2$ with diameter at most $\eps$ where $\cf(\gamma_1,T)$ and $\cf(\gamma_2,T)$ are implicit functions.  
Let $A_i$ be a closed annulus with smooth boundary such that $\cf(\gamma_i,T)$ is in the interior of $A_i$, and $\partial A_i \cap R_0$ is a pair of curves that are implicit functions, and $A_1\cap A_2 \cap \partial R_0 = \emptyset$.  We could find such an $A_i$ by constructing a tubular neighborhood about $\cf(\gamma_i,T)$ such that in $R_0$ the fibers are along curves of longitude. 
Let $R_i$ be the closure of $A_i \setminus R_0$.  Let $M$ be the disjoint union of $R_1$ and $R_2$ together with $R_0$ maintaining the intersections $R_1\cap R_0$ and $R_2 \cap R_0$.
Observe that $M$ is a smooth oriented Riemannian surface. 

%, and additional disjoint surface patches added to each corner to make the boundary of $R_3$ smooth without changing the topology. 
%with either 1 boundary component in the case where $\cf(\gamma_1,T)$ and $\cf(\gamma_2,T)$ alternate around the boundary of $R_0$, or 3 boundary components otherwise.
%Let $M$ be the union of $N$ with a polygonal surface patch for each boundary component of $N$ such that the boundary of each polygon is identified with the corresponding boundary component so as to create a smooth oriented surface.  
%Now $R_4$ is an oriented Riemannian surface.

Since $\cf(\gamma_i,T)$ is in the interior of the annulus $A_i$, $\cf(\gamma_i,T)$ is bounded away from the boundary of $A_i$, so for some $t_0 < T$ sufficiently close to $T$, the curve $\cf(\gamma_i,t)$ is in the interior of $A_i$ for all $t \in [t_0,T]_\mb{R}$.  Let $\cf_{M}(\gamma_i)$ be the evolution by curvature flow in $M$ starting from $\cf_{M}(\gamma_i,0) = \cf(\gamma_i,t_0)$.  Then, $\cf_{M}(\gamma_i,t)$ coincides with $\cf(\gamma_i,t+t_0)$ as a curve in $A_i$.  Hence, the curves $\cf_{M}(\gamma_1,t)$ and $\cf_{M}(\gamma_2,t)$ do not intersect for $t<T-t_0$ as they evolve, but they do intersect at $t=T-t_0$, which contradicts Angenet's theorem. 
Thus, our assumption cannot hold, so for each $k$ there is some $t_k$ and a point $q_k \in \cf(\gamma_1,t_k) \cap \cf(\gamma_2,t_k)$ such that $t_k \to T$ and $q_k \to p_\mr{T}$. 

Let $t_0$ be an earlier time than $T$ such that there are no tangential intersections during the interval $(t_0,T)_\mb{R}$.	
As we already showed above, we can partition the intersection graphs $\graph(\cf(\gamma_1)\cap\cf(\gamma_2)) = \{(t,\cf(\gamma_1,t)\cap \cf(\gamma_2,t)) : t \in (t_0,T)_\mb{R}\}$ into finitely many trajectories.  One of these trajectories, which we call $p$, must contain infinitely many of the points $q_k$,
so $p(t_k) \to p_\mr{T}$.
If there were some other sequence of times $r_k \to T$ from below such that $p(r_k) \not\to p_\mr{T}$, then such sequences of $p(t)$ would converge to infinitely many different points where $\cf(\gamma_1,T)$ and $\cf(\gamma_2,T)$ intersect, which would violate Angenet's theorem. 
Specifically, we could restrict $p(r_k)$ to a convergent subsequence bounded away from $p_\mr{T}$ by some $b> 0$, and for each $m > 1$ there would be a sequence of times $s_k \to T$ such that $p(s_k)$ is distance $b/m$ from $p(t_k)$ and converges to a point of intersection that is distance $b/m$ from $p_\mr{T}$. 
Since sequential continuity implies continuity, we have $p(t) \to p_\mr{T}$ as $t \to T$ from below.

We now have in either case where $p_\mr{T}$ is a point of transversal or tangential intersection, there is a trajectory $p$ such that $p(T) = p_\mr{T}$ and $p(t) \in \cf(\gamma_1,t)\cap \cf(\gamma_2,t)$ for all $t \in [t_0,T]_\mb{R}$ where $t_0$ is the precious time that a transversal intersection occurs or is 0 in the case where no transversal intersection occurs before $T$.  
Since the number of intersections is finite and can only decrease at a time of common tangency \cite[Theorem 1.3]{angenent1991parabolic}, only finitely many times of tangency occur. 
Therefore, by induction on the number of times of transversal intersections after time $t$, the trajectory $p$ can be extended so that $p(t) \in \cf(\gamma_1,t)\cap \cf(\gamma_2,t)$ for all $t \in (0,T]_\mb{R}$. 
It only remains to show that the trajectory converges as $t \to 0$.

Suppose for the sake of contradiction that $p(t)$ does not converge as $t \to 0$. 
Then, we could find two sequences of times $r_k,t_k \to 0$ that are bounded apart. 
Since the sphere is compact, we may assume that $p(r_k)$ and $p(t_k)$ respectively converge to points $p_\mr{r},p_\mr{t} \in (\gamma_1\cap\gamma_2)$; otherwise restrict to a convergent subsequences. 
Furthermore, we may choose the sequences so that they are alternating, i.e.\ $r_1>t_1>r_2>t_2>\dots$. 
Let $d> 0$ be the minimum distance between points where $\gamma_1$ and $\gamma_2$ intersect.
Let $s_k$ be the first time after $r_k$ such that $p(s_k)$ is distance $d/2$ from $p(r_k)$.
Again, we may assume by compactness that $p(s_k)$ converges to a point $p_\mr{s} \in (\gamma_1\cap\gamma_2)$.
Then, $p_\mr{s}$ would be a point of intersection that is distance $d/2$ from $p_\mr{r}$, contradicting our choice of $d > 0$ as the minimum such distance. 
Thus, $p(t)$ must converge as $t \to 0$,
and therefore there is a trajectory $p : [0,T]_\mb{R} \to \sphere^2$ that remains in the intersection of the curves as they evolve. 
\end{proof}

\begin{lemma}\label{lemma-crossingcontinuous}
Let $\gamma,\eta \in \bis$ be a pair of curves that intersect at only 2 points, and let $x$ be one of the points of intersection.
Then, there is a unique continuous trajectory $p(\gamma,\eta,x) : [0,\infty) \to \sphere^2$ such that for all $t$ we have $p(\gamma,\eta,x;t) \in (\cf(\gamma_1,t)\cap\cf(\gamma_2,t))$ and $p(\gamma,\eta,x;0) = x$. 
Furthermore, $p(\gamma,\eta,x,t)$ is continuous as a function of $\gamma$ and $\eta$ in Fréchet distance and $x$ and $t$.
\end{lemma}

\begin{proof}
The curves $\cf(\gamma,t)$ and $\cf(\eta,t)$ must intersect at exactly two points throughout their evolution for $t<\infty$, since they cannot intersect in more than 2 by Angenet's theorem, and they cannot intersect in fewer than 2, since they remain as bisectors throughout the deformation, which implies that one curve cannot be properly contained in the region on one side of the other curve. 
By Lemma \ref{lemma-crossingpaths}, for each $T \in (0,\infty)_\mb{R}$, 
there is a pair of trajectories $p_0,p_1$ that are always in the intersection of the evolving curve and that arrive at the two distinct intersection point at time $T$. 

We claim that the points $p_0(t)$ and $p_1(t)$ are distinct for each $t \leq T$. 
Suppose not, and let $t_0$ be the last time where $p_0(t_0)$ and $p_1(t_0)$ coincide. 
Note that there must be such a last time $t_0$ since the $p_i$ are continuous and so the $p_i$ coincide on a closed set, and that $t_0<T$. 
Observe that for $t > t_0$, $\cf(\gamma,t)$ consists of a pair of arcs connecting the points $p_0(t)$ and $p_1(t)$.  If either of these arcs were to shrink to a point as $t \to t_0$ from above, then $\cf(\gamma,t_0)$ would be properly contained in the region on one side of $\eta$, which is impossible, since both curves are bisectors.  On the other hand, if neither were to shrink to a point, then $\cf(\gamma,t_0)$ would consist of pair of arcs from $p_0(t_0) = p_1(t_0)$ to itself, which is not even a simple closed curve, so this is also impossible.  In either case were get a contradiction, so the trajectories never coincide. 
Let $p(\gamma,\eta,x) : [0,\infty)_\mb{R} \to \sphere^2$ be the trajectory starting from $p(\gamma,\eta,x;0) = x$. It only remains to show continuity. 

Consider $\gamma_k,\eta_k \in \bis$ and $x_k \in \gamma_k\cap\eta_k$ and $t_k \in [0,\infty)_\mb{R}$ for $k \in \{1,\dots,\infty\}$ such that $\gamma_k \to \gamma_\infty$ and $\eta_k \to \eta_\infty$ in Fréchet distance and $x_k \to x_\infty$ and $t_k \to t_\infty$. 
Let $p_k = p(\gamma_k,\eta_k,x_k;t_k)$ and let $q_k \in \cf(\gamma_k,t_k)\cap \cf(\eta_k,t_k)$ be the other point in the intersection. 
By Lemma \ref{lemma-flowconvergesinhausdorff}, 
$\cf(\gamma_k,t_k) \to \cf(\gamma_\infty,t_\infty)$ and $\cf(\eta_k,t_k) \to \cf(\eta_\infty,t_\infty)$ 
in Hausdorff distance.

We claim that the accumulation set of $p_k$ is contained in the pair of points 
$\{p_\infty,q_\infty\}$.
That is, every convergent subsequence converges to either $p_\infty$ or $q_\infty$.
%\cf(\gamma_\infty,t_\infty)\cap \cf(\eta_\infty,t_\infty)$.
Consider $\eps>0$, and let $\delta = \delta(\eps)$ be the distance between $\cf(\gamma_\infty,t_\infty) \setminus X_\eps$ and $\cf(\eta_\infty,t_\infty) \setminus X_\eps$ where $X_\eps = \{p_\infty,q_\infty\} \oplus\eps$.
Observe that $\delta > 0$, since the closures of these sets are compact and disjoint. 
%Then, $\delta > 0$ and $\delta(0) \to 0$ as $\eps \to 0$.
Observe also that $(\cf(\gamma_\infty,t_\infty) \oplus \delta/3) \setminus X_\eps$ and $(\cf(\eta_\infty,t_\infty) \oplus \delta/3) \setminus X_\eps$ are disjoint.
For all $k$ sufficiently large, $\cf(\gamma_k,t_k)$ is within $\delta/3$ of $\cf(\gamma_\infty,t_\infty)$ 
in Hausdorff distance, and analogously for $\eta_k$. 
Hence, $\cf(\gamma_k,t_k)$ and $\cf(\eta_k,t_k)$ only intersect within $X_\eps$, 
so $p_k$ is within $\eps$ of $\{p_\infty,q_\infty\}$.
Since this holds for all $\eps >0$, the accumulation set is contained in the pair $\{p_\infty,q_\infty\}$. 

Suppose for the sake of contradiction that $q_\infty$ were an accumulation point of $p_k$. 
With $\gamma_k,\eta_k,x_k$ fixed, 
we could choose a sequence $t_k$ such that $t_\infty$ is the infimum of times for which there exists some convergent sequence of times $t_k \to t_\infty$ such that $p_k \to q_\infty$.
Since $p(\gamma_k,\eta_k,x_k;0) = x_k \to x_\infty = p(\gamma_\infty,\eta_\infty,x_\infty;0)$, we know there is some convergent sequence of times $r_k \to r_\infty$ such that $p(\gamma_k,\eta_k,x_k;r_k) \to p(\gamma_\infty,\eta_\infty,x_\infty;r_\infty)$.
By definition of $t_\infty$, we could choose a sequence $r_k\leq t_\infty$ so that $r_\infty = t_\infty$. 
More precisely, we could choose $r_k < t_\infty$ in the case where $t_\infty > 0$, or $r_k = 0$ in the case where $t_\infty = 0$.

We would then have a sequence of trajectories 
$p(\gamma_k,\eta_k,x_k;t)$ on $t \in [r_k,t_k]_\mb{R}$ 
with end-points respectively converging to $p_\infty$ and $q_\infty$. 
Since $p(\gamma_k,\eta_k,x_k)$ is a continuous function, we could choose a sequence of times $s_k \in [r_k,t_k]_\mb{R}$ such that $p(\gamma_k,\eta_k,x_k;s_k)$ is equidistant from the end-points. 
%$p(\gamma_k,\eta_k,x_k;r_k)$ and $p_k = p(\gamma_k,\eta_k,x_k;t_k)$. 
Since $\sphere^2$ is compact, we may assume that $p(\gamma_k,\eta_k,x_k;s_k)$ converges to a point $z$;
otherwise restrict to a convergent subsequence.
Since 
$\cf(\gamma_k,s_k) \to \cf(\gamma_\infty,t_\infty)$ and $\cf(\eta_k,s_k) \to \cf(\eta_\infty,t_\infty)$
in Hausdorff distance, $z$ would be a third point of intersection of $\cf(\gamma_\infty,t_\infty)$ and $\cf(\eta_\infty,t_\infty)$ that is equidistant from $p_\infty$ and $q_\infty$, but that is impossible.
Thus, $q_\infty$ cannot be an accumulation point of $p_k$.  
Therefore, $p_k \to p_\infty$, so $p$ is continuous.
\end{proof}

\vbox{
\begin{lemma}\label{lemma-crossingcurve}
Let $\gamma \in \bis$, $T \in (0,\infty)_\mb{R}$, and $z \in \cf(\gamma,T)$. 
Then, there is a curve $\zeta \in \bis$ that intersect $\gamma$ at exactly 2 points, is smooth away from $\gamma$, and evolves to pass through $z$ at time $T$, i.e.\ $z \in \cf(\zeta,T)$.
\end{lemma}
}

\begin{proof}
Let $\mathbf{D}$ denote the unit disk, and let $\mu$ denote surface area on the sphere. 
%Given a pair of directed Jordan curves $\gamma_0$ and $\gamma_1$, we say that $\gamma_1$ crosses $\gamma_0$ leftward at a point $p$ when $\gamma_1$ crosses from the right side of $\gamma_0$ to the left side, and we say rightward in the opposite case.

Fix %a parameterization $\widehat{\gamma}:\sphere^1 \to \gamma$ and 
a pair of internally conformal homeomorphisms $\phi_+,\phi_-: \mathbf{D} \to \sphere^2$ that map the closed unit disk to the closed region on either side of $\gamma$ as in Carathéodory's mapping theorem \cite{pommerenke1992boundary},
and direct $\gamma$ so $\phi_+$ is on the left.
%, which we call the positive side of $\gamma$.

For $x,y \in \gamma$, let
\[ \alpha(x,y) = \phi_+\left(\left[\phi_+^{-1}(x),\phi_+^{-1}(y)\right]_\mb{C}\right) \cup \phi_-\left(\left[\phi_-^{-1}(x),\phi_-^{-1}(y)\right]_\mb{C}\right) \]
be directed so that $\alpha(x,y)$ crosses $\gamma$ leftward at $x$.  
Let $A(x,y)$ be the region to the left of $\alpha(x,y)$,  
%oriented from $x$ to $y$ though $\phi_+(\mathbf{D})$,
and let $a(x,y) = \mu(A(x,y))$ be the area of $A(x,y)$. 

We claim that $a(x,y)$ is continuous.
Consider $x_k, y_k  \in \gamma$ such that $x_k\to x_\infty$ and $y_k \to y_\infty$, and let $a_k$ be the area of $A_k=A(x_k,y_k)$.
For $n \geq K$ we have 
\[ \mu \left(\bigcap_{k=K}^\infty A_k \right) \leq \mu \left( A_n \right) = a_n \leq \mu\left(\bigcup_{k=K}^\infty A_k\right), \text{ so}  \]
\[ 
\mu \left(\bigcap_{k=K}^\infty A_k \right) \leq \inf_{k\geq K} a_k,   
\quad 
\mu\left(\bigcup_{k=K}^\infty A_k\right) \geq \sup_{k\geq K} a_k.  
\]
\[ A_\infty  \subseteq \bigcup_{K=1}^\infty \bigcap_{k=K}^\infty A_k \subseteq \bigcap_{K=1}^\infty \bigcup_{k=K}^\infty A_k \subseteq \overline{A_\infty}. \]
and since the boundary of $A_\infty$ has area 0, $\mu \left( \overline{A_\infty} \right) = \mu \left( A_\infty \right) = a_\infty$.
By continuity from above and below, 
\begin{gather*} 
a_\infty = \mu\left(\bigcap_{K=1}^\infty \bigcup_{k=K}^\infty A_k \right)
= \lim_{K\to\infty} \mu \left( \bigcup_{k=K}^\infty A_k \right) \geq \limsup_{k \to \infty} a_k, 
\\
a_\infty = \mu \left( \bigcup_{K=1}^\infty \bigcap_{k=K}^\infty A_k \right)  
= \lim_{K \to \infty} \mu \left( \bigcap_{k=K}^\infty A_k \right) \leq \liminf_{k\to\infty} a_k, 
\end{gather*}
so $a_k \to a_\infty$, which means the claim holds.

For $y'$ to the left of $\alpha(x,y)$ directed as above, $A(x,y) \supsetneq A(x,y')$ and $A(x,y) \setminus A(x,y')$ has positive measure, so $a(x,y) \geq a(x,y')$, which means $a(x)$ is strictly increasing along $\gamma$ from 0 to $4\pi$.  Thus, for each $x \in \gamma$ there is a unique $\upsilon_0(x) \in \gamma$ such that 
%by the intermediate value theorem, there is a map $\upsilon: \gamma(0) \to \gamma(0)$ such that 
$a(x,\upsilon_0(x)) = 2\pi$.

We claim $\upsilon_0$ is continuous.  
Suppose not. Then by compactness of $\gamma$, there is a sequence $\upsilon_0(x_k) \to y \neq \upsilon_0(x_\infty)$.  Also, $a(x_n,\upsilon_0(x_k)) \to a(x_\infty,y)$, and $a(x_n,\upsilon_0(x_k)) = 2\pi$, so $a(x_\infty,y) = 2\pi$, but that contradicts the uniqueness of $\upsilon_0(x_\infty)$.
Thus, the claim holds.

Let $\eta(x) = \alpha(x,\upsilon_0(x))$, 
and let $p(x;t) \in \cf(\gamma,t)\cap\cf(\eta(x),t)$ be the trajectory starting from $p(x;0) = x$ as in the Lemma \ref{lemma-crossingcontinuous}. 
Since $\phi_+$ is a homeomorphism, we have $\phi_+^{-1}(x_k) \to \phi_+^{-1}(x_\infty)$ and $\phi_+^{-1}(\upsilon_0(x_k)) \to \phi_+^{-1}(\upsilon_0(x_\infty))$, so $[\phi_+^{-1}(x_k),\phi_+^{-1}(\upsilon_0(x_k))]_\mathbf{D} \to [\phi_+^{-1}(x_\infty),\phi_+^{-1}(\upsilon_0(x_\infty))]_\mathbf{D}$ in Fréchet distance, and likewise for $\phi_-$. 
By the Heine-Cantor theorem, $\phi_+$ and $\phi_-$ are uniformly continuous, so 
$\eta(x_k) \to \eta(x_\infty)$ in Fréchet distance. 
Therefore, by Lemma \ref{lemma-crossingcontinuous}, $p(x_k;t_k) \to p(x_\infty;t_\infty)$. 
Hence, $p$ is continuous. 

The map $p$ defines a homotopy from the identity map on $\gamma$ to the map $p(\cdot,T) : \gamma \to \cf(\gamma,T)$, 
so $p(\cdot,T)$ winds once around $\cf(\gamma,T)$, which implies there is some $x \in \gamma$ such that $\chi(x,T) = z$.
Thus, $\zeta = \eta(x)$ has the desired properties.
\end{proof}

To prove Theorem \ref{theoremFlowConvergesInFrechet} in the finite time case, we will use another theorem of Lauer's, which for a sequence of curves $\gamma_k$ converging to $\gamma_\infty$ in Fréchet distance and $t_k \to t_\infty > 0$, 
gives an upper bound on the length of $\cf(\gamma_k,t_k)$ independent of $k$ for all $k$ sufficiently large \cite[Theorem 1.3]{lauer2016evolution}.
This upper bound is in terms of $r$-multiplicity, a quantity used in the proof of the infinite time case of the Theorem \ref{theoremFlowConvergesInFrechet}.

\begin{proof}[Proof of Theorem \ref{theoremFlowConvergesInFrechet} for $t_\infty < \infty$.]
Let $\gamma_k$ satisfy the hypothesis of the theorem, and assume for the sake of contradiction that $\cf(\gamma_k,t_k) \not\to \cf(\gamma_\infty,t_\infty)$ in Fréchet distance.
By a theorem of Lauer, $\cf(\gamma_k,t_k)$ is smooth and has length bounded by some constant for all $k$ sufficiently large \cite[Theorem 1.3]{lauer2016evolution}, provided that $t_\infty > 0$.  Therefore, the sequence of constant speed parameterizations $\psi_k : \sphere^1 \to \cf(\gamma_k,t_k)$ is uniformly equicontinuous, so by the Arzelà-Ascoli theorem, we can restrict to a sequence that converges uniformly to a map $\psi_\infty$.
Moreover, by Lemma \ref{lemma-flowconvergesinhausdorff}, the range of the limit $\psi_\infty$ is $ \cf(\gamma_\infty,t_\infty)$. 
We can also let $\omega: \sphere^1 \to \cf(\gamma_\infty,t_\infty)$ be a constant speed parameterization.
%Then, $\omega^{\inv}\circ\psi_\infty: \sphere^1 \to \sphere^1$ lifts 
%by $\theta \mapsto e^{\mathfrak{i}\theta}$
Let $\phi :[0,2\pi]_\mb{R} \to \mb{R}$ by 
\[\phi(\theta) = -\mathfrak{i}\log(\omega^{-1}\circ \psi_\infty (e^{\mathfrak{i}\theta})) ,\] 
which is just $\omega^{-1}\circ \psi_\infty$ lifted by the standard parameterization of the circle by angle.

Since the map $(\theta \mapsto \omega^{-1} \circ \psi_\infty (e^{\mathfrak{i}\theta}))$ is periodic, $\phi(2\pi)-\phi(0)$ is a multiple of $2\pi$.  
We may choose the direction of $\omega$ so that $\phi(2\pi)-\phi(0) \geq 0$.
If we had $\phi(2\pi)-\phi(0) = 0$, then the region on one side of $\cf(\gamma_k,t_k)$ would converge to a subset of $\cf(\gamma_\infty,t_\infty)$, but $\gamma_k$ continues to be a bisector as it evolves, so that cannot happen. 
Hence, $\phi(2\pi)-\phi(0) \neq 0$.
If we had $\phi(2\pi)-\phi(0) > 2\pi$, then $\cf(\gamma_k,t_k)$ would wind more than once around a tubular neighborhood of $\cf(\gamma_\infty,t_\infty)$, which is impossible for a simple closed curve, so $\phi(2\pi)-\phi(0) = 2\pi$.

%Observe that the total variation of $\phi$ is at least $2\pi$.  If the total variation of $\phi$ were equal to $2\pi$, then $\phi$ would be non-decreasing.  In that case,
If $\phi$ were weakly increasing, then $\phi$ would be the limit of some sequence of strictly increasing functions $\phi_k \to \phi$, and we would have homeomorphisms 
$\tau_k : \sphere^1 \to \cf(\gamma_\infty,t_\infty)$ given by  
$\tau_k(x) = \omega(e^{\mathfrak{i}\phi_k(-\mathfrak{i}\log(x))}))$ that converge to $\psi_\infty$, but then the Fréchet distance between $\cf(\gamma_k,t_k)$ and $\cf(\gamma_\infty,t_\infty)$ would be bounded by $\sup_{x \in \sphere^1}\|\psi_k(x)-\tau_k(x)\| \to 0$, 
which contradicts our assumption that $\cf(\gamma_k,t_k) \not\to \cf(\gamma_\infty,t_\infty)$ in Fréchet distance.
Hence, $\phi$ must decrease somewhere, i.e.\ there is $w < y$ such that $\phi(y) < \phi(w)$, and we may choose the period of $\phi$ so that $\phi(w)$ is between $\phi(0)$ and $\phi(2\pi)=\phi(0) +2\pi$, 
and may choose $w,y$ arbitrarily close together.  Then, there are $v<w<x<y<z$ where $\phi(v) = \phi(y) < \phi(x) < \phi(w) = \phi(z)$.

Let $ v_k = \psi_k(e^{\mathfrak{i}v})$, $ w_k = \psi_k(e^{\mathfrak{i}w})$, $ x_k = \psi_k(e^{\mathfrak{i}x})$, $ y_k = \psi_k(e^{\mathfrak{i}y})$, $ z_k = \psi_k(e^{\mathfrak{i}z})$.
The situation so far is that as we traverse $\cf(\gamma_k,t_k)$, we pass through $ v_k, w_k, x_k, y_k, z_k$ in that order, and these points respectively converge to $ y_\infty, w_\infty, x_\infty, y_\infty, w_\infty$; see Figure \ref{figure-Frechet-converge}.
By Lemma \ref{lemma-crossingcurve}, there is a bisector $\zeta$ that intersects $\gamma$ at exactly 2 points and such that $\cf(\zeta,t_\infty)$ intersects $\cf(\gamma_\infty,t_\infty)$ at $ x_\infty$. 
Also, we can choose $w,y$ sufficiently close together so that the 2 points $\cf(\zeta,t_\infty) \cap \cf(\gamma_\infty,t_\infty)$ are not on the same arc from $ w_\infty$ to $ y_\infty$.
That is, $ w_\infty$ and $ y_\infty$ are on opposite sides of $\cf(\zeta,t_\infty)$.

Since level-set flow gives a solution to the curvature flow problem, we can find a nested sequence of smooth annuli $A_{1,0} \supset A_{2,0} \supset \cdots$ such that $\cf(\zeta,t) = \bigcap_{i=1}^\infty A_i(t)$ where the boundary of $A_i(t)$ evolves by curvature flow starting $A_i(0) = A_{i,0}$   \cite{lauer2016evolution}. 
Choose one of the annuli $A = A_i$ such that $A(t)$ remains and annulus up to time $t=2t_\infty$ 
and close enough to $\zeta$ that $A(t_\infty)$ does not contain $ w_\infty$ or $ y_\infty$; see Figure \ref{figure-Frechet-converge}.
Let $\alpha$ and $\beta$ be the curves on the boundary of $A(0)$.
Note that 
$\cf(\alpha,t)$ and $\cf(\beta,t)$ are on either side of $\cf(\zeta,t)$ for all $t \in [0,2t_\infty]_\mb{R}$.
%since $ w_\infty$ and $ y_\infty$ are on either side of $\zeta(t_\infty)$, 
%$\alpha(t_\infty)$ and $\beta(t_\infty)$ are on either side of $\zeta(t_\infty)$
%$\zeta(t_\infty)$ winds once around $A(t_\infty)$, so $\zeta(t)$ winds once around $A(t)$ for all $t \in [0,t_\infty]_\mb{R}$.
Hence, each arc of $\gamma_\infty$ subdivided by $\alpha\cup\beta$ that goes from a point on $\alpha$ to a point on $\beta$ must cross $\zeta$, and since $\gamma_\infty$ meets $\zeta$ at only 2 points, there can be at most 2 such arcs.

Since $\alpha,\zeta,\beta$ are compact and pairwise disjoint, we may choose $\eps_1> 0 $ small enough that $\alpha\oplus\eps_1,\zeta\oplus\eps_1,\beta\oplus\eps_1$ are pairwise disjoint. 
Let us choose $k$ sufficiently large that $\gamma_k$ is at most Fréchet distance $\eps_1$ from $\gamma_\infty$, and $ v_k$, $ w_k$, $ y_k$, and $ z_k$ are each outside of $A(t_k)$.
To see that the later condition can be satisfied, recall that these points each approach either $ w_\infty$ or $ y_\infty$, which are bounded away from $A(t_\infty)$, and $A(t_k) \to A(t_\infty)$ in Hausdorff distance. 

Consider an arc $\xi$ of $\gamma_k \cap A(0)$ from $\alpha$ to $\beta$.
Since $\gamma_k$ is withing Fréchet distance $\eps_1$ from there is some map $\psi : \xi \to \gamma_\infty$ such that $\|\psi(x) -x\| \leq \eps_1$ for all $x \in \xi$, so $\psi(\xi)$ must be an arc of $\gamma_\infty$ from $\alpha\oplus\eps_1$ to $\beta\oplus\eps_1$.
Hence, $\psi(\xi)$ must intersect $\zeta$, and since $\gamma_\infty$ only intersects $\zeta$ at 2 points, 
there are only 2 arcs of $\gamma_k \cap A(0)$ from one boundary of $A(0)$ to the other. 
Therefore, we can find a curve $\eta$ of area 0 that winds once around the interior of $A(0)$ and intersects $\gamma_k$ at only 2 points; see Figure \ref{figure-Frechet-converge}.

\begin{figure}
\includegraphics{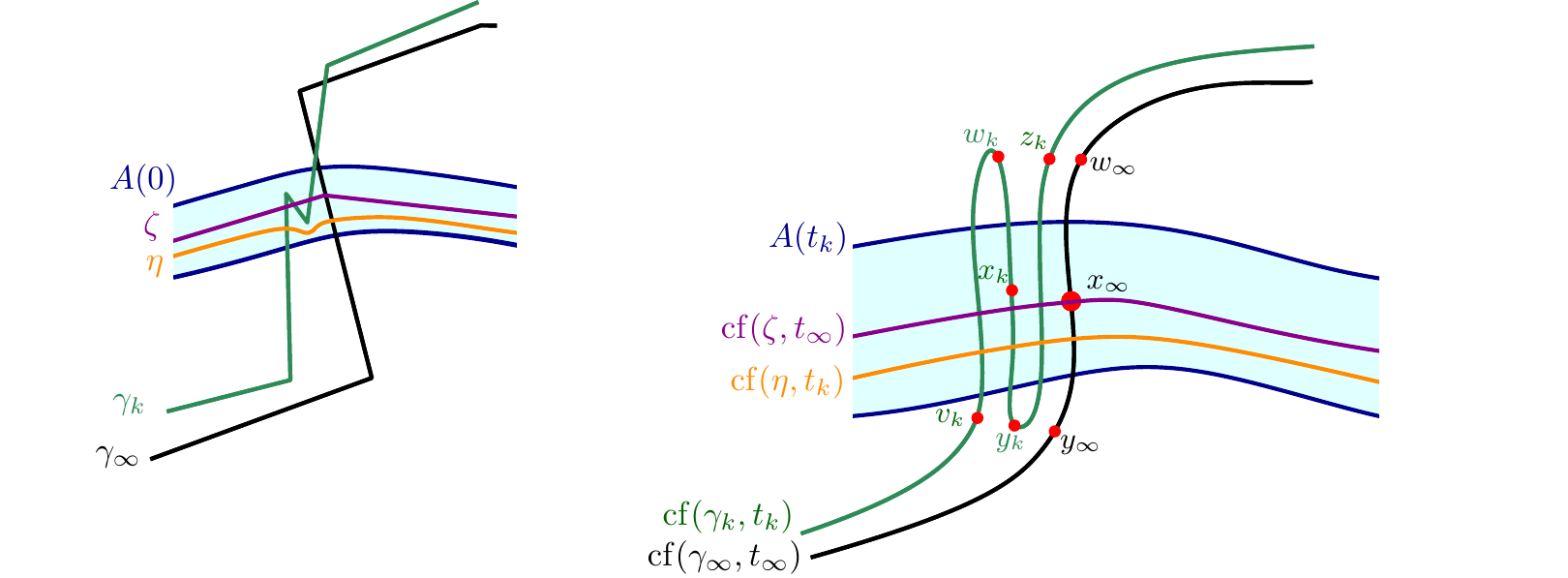}
\caption{
Assuming $\cf(\gamma_k,t_k) \to \cf(\gamma_\infty,t_\infty)$ in Hausdorff distance but not in Fréchet distance, 
$\gamma_k$ evolves to pass close by a point $x_\infty$ three times, and the curve $\zeta$ only intersects $\gamma_\infty$ twice and evolves to pass through $x_\infty$.
The curve $\eta$ crosses $\gamma_k$ only twice, but evolves to cross $\cf(\gamma_k,t_k)$ more than twice, contradicting Angenet's theorem.
}
\label{figure-Frechet-converge}
%Since $\gamma_k \to \gamma_\infty$ in Fréchet distance, $\cf(\gamma_k,t_k) \to \cf(\gamma_\infty,t_\infty)$ in Hausdorff distance.  Assuming $\cf(\gamma_k,t_k) \not\to \cf(\gamma_\infty,t_\infty)$ in Fréchet distance, 
%there is a point $x_\infty$ that is the limit of 3 points bounded apart on $\gamma_k$, and a curve $\zeta$ that only intersects $\gamma_\infty$ twice and evolves to pass through $x_\infty$ at time $t_\infty$.
%}
%\includegraphics{figures/Frechet-converge3}
%\caption{There is a curve $\eta$ that crosses $\gamma_k$ only twice, but evolves to cross $\cf(\gamma_k,t_k)$ more than twice, contradicting Angenet's theorem.}
%\label{figure-Frechet-converge-bottom}
\end{figure}

By Angenet's theorem, $\cf(\eta,t)$ remains disjoint from $\cf(\alpha,t)$ and $\cf(\beta,t)$, and crosses $\cf(\gamma_k,t)$ at most twice, for as long as a solution to curvature flow exists \cite{angenent1991parabolic}, 
and by Lauer's theorem, the evolving curves $\cf(\alpha,t),\cf(\beta,t)$ on either side of $\eta$ ensures that a unique solution exists for $\cf(\eta,t)$ up to time at least  $2t_\infty$ \cite{lauer2016evolution}. 

Since $\cf(\gamma_k,t_k)$ passes though $ v_k, w_k, y_k, z_k$ in that order, and $ v_k$ and $ y_k$ are on one side of $A(t_k)$, and $ w_k$ and $ z_k$ are on the other side of $A(t_k)$, $\cf(\gamma_k,t_k)$ must cross from one side of $A(t_k)$ to the other at least 4 times.
Since $\cf(\eta,t)$ never intersects the boundary of $A(t)$, $\cf(\eta,t_k)$ winds once around $A(t_k)$.
Therefore, $\cf(\eta,t_k)$ intersects $\cf(\gamma_k,t_k)$ at least at 4 points, but $\eta$ intersects $\gamma_k$ at only 2 points, which contradicts Angenet's theorem that the number of intersection points does not increase; see Figure \ref{figure-Frechet-converge}.  Thus, our assumption that $\gamma_k \not\to \gamma_\infty$ in Fréchet distance cannot hold
\end{proof}

To deal with the case where $t_\infty = \infty$, we will use $r$-multiplicity. 
For $0 < r < \nicefrac\pi2$, 
the $r$-multiplicity $M_{r,g}$ of a Jordan curve $\gamma$ at a great circle is the number of connected components of $\gamma\cap (g\oplus 2r)^\circ$ that intersect $g\oplus r$ in the angular metric on the sphere.
Lauer showed for all Jordan curves $\gamma$ that $M_{r,g}(\cf(\gamma,t))$ is non-increasing \cite[Lemma 7.3]{lauer2016evolution}, and for all sequences $\gamma_k \to \gamma_\infty$ in Fréchet distance that 
\[ \limsup_{k \to \infty} M_{r,g}(\gamma_k) \leq M_{r,g}(\gamma_\infty) \]
\cite[Lemma 7.4]{lauer2016evolution}.

\begin{proof}[Proof of Theorem \ref{theoremFlowConvergesInFrechet} for $t_\infty = \infty$.]
Here we use a similar argument to the case where $t_\infty$ was finite, except the role of the annuli $A_i(t)$ in those arguments will be replaced by $r$-multiplicity.
Consider sequences $\gamma_k \to \gamma_\infty$ in Fréchet distance and $t_k \to \infty$. 

Let us first consider the case where $\gamma_\infty$ is a great circle.
Then, $\cf(\gamma_\infty,t) = \gamma_\infty$ is an unchanging great circle for all time. 
Let us refer to $\gamma_\infty$ as the ``equator'' and refer to the semicircles that cross perpendicularly at their midpoint as ``meridians.'' 

Consider $\eps > 0$.  
Since $\gamma_k \to \gamma_\infty$, for all $k$ sufficiently large, 
$\gamma_k \subset \gamma_\infty \oplus \eps$, so 
the $\eps$-multiplicity of $\gamma_k$ at $\gamma_\infty$ is 1, i.e.\ $M_{\eps,\gamma_\infty}(\gamma_k) = 1$.  
Observe that $M_{\eps,\gamma_\infty}(\cf(\gamma_k,t))$ can never vanish, since $\gamma_\infty$ and $\cf(\gamma_k,t)$ are bisectors, and therefore cannot be disjoint.  Since the $\eps$-multiplicity of a curve evolving by level-set flow is non-increasing \cite{lauer2016evolution}, we must have $M_{\eps,\gamma_\infty}(\cf(\gamma_k,t)) = 1$, so $\cf(\gamma_k,t) \subset \gamma_\infty\oplus2\eps$.
Furthermore, for $k$ sufficiently large $\gamma_k$ winds once around the annulus $\gamma_\infty\oplus2\eps$, and since $\cf(\gamma_k)$ gives a homotopy from $\gamma_k$ to $\cf(\gamma_k,t_k)$, the curve $\cf(\gamma_k,t_k)$ must also wind once around $\gamma_\infty\oplus2\eps$. 
By Alexander duality,  
every curve through the annulus $\gamma_\infty\oplus2\eps$ going from one boundary circle of the annulus to the other must intersect the curve $\cf(\gamma_k,t_k)$. 
In particular, $\cf(\gamma_k,t_k)$ must intersect the arc of each meridian though the annulus, 
which implies that $\gamma_\infty \subset \gamma_k \oplus2\eps$.  Hence, $\gamma_k$ and $\gamma_\infty$ are within Hausdorff distance $2\eps$. 
Letting $\eps \to 0$, we get that $\cf(\gamma_k,t_k) \to \gamma_\infty$ in Hausdorff distance. 

Suppose for the sake of contradiction that $\cf(\gamma_k,t_k)$ does not converge to $\gamma_\infty$ in Fréchet distance. 
Then, like in the case where $t_\infty$ was finite, 
as we traverse $\cf(\gamma_k,t_k)$, we would pass through points $ v_k, w_k, x_k, y_k, z_k$ that respectively converge to points $ y_\infty, w_\infty, x_\infty, y_\infty, w_\infty$ on $\gamma_\infty$.
Let $3r$ be the minimum distance between the points $w_\infty, x_\infty, y_\infty$. 
Let $\mu$ be the meridian that passes through $x_\infty$ on the equator and $\zeta$ be the great circle containing $\mu$.
Since $\gamma_k \to \gamma_\infty$ in Fréchet distance, $\gamma_k$ is within distance $r$ for all $k$ sufficiently large, which implies that the $r$-multiplicity of $\gamma_k$ at $\zeta$ is 2, i.e.\ $M_{r,\zeta}(\gamma_k) = 2$. 
Since $r$-multiplicity is non-increasing, and $\cf(\gamma_k,t_k)$ must intersect each meridian in the annulus $\gamma_\infty\oplus2\eps$, we must have $M_{r,\zeta}(\cf(\gamma_k,t_k)) = 2$. 
However, for all $k$ sufficiently large each of the points $v_k,w_k,y_k,z_k$ on $\cf(\gamma_k,t_k)$ would be within distance $r$ of the respective limit point $ y_\infty, w_\infty, y_\infty, w_\infty$, which are on opposite sides of the annulus $\zeta\oplus2r$. 
Therefore, $\cf(\gamma_k,t_k)$ would have to cross through the annulus $\zeta\oplus2r$ at least 4 times, which would imply that $M_{r,\zeta}(\cf(\gamma_k,t_k)) \geq 4$, which is impossible.
Thus, $\cf(\gamma_k,t_k) \to \gamma_\infty$ in Fréchet distance, provided that $\gamma_\infty$ is a great circle.

We now prove the general case for $\gamma_\infty \in \bis$, not restricted to being a great circle.
Suppose for the sake of contradiction that $\cf(\gamma_k,t_k) \not\to \cf(\gamma_\infty,\infty)$ in Fréchet distance.  
Then, we could assume that $\cf(\gamma_k,t_k)$ is bounded away from $\cf(\gamma_\infty,\infty)$; otherwise choose an appropriate subsequence. 
Since $\cf(\gamma_\infty,t)$ converges to the great circle $\cf(\gamma_\infty,\infty)$ as $t \to \infty$, we may choose $s_n$ sufficiently large that $\cf(\gamma_\infty,s_n)$ is at most Fréchet distance $\frac1{2n}$ from $\cf(\gamma_\infty,\infty)$.  Since the finite time case of the Theorem \ref{theoremFlowConvergesInFrechet} holds, $\cf(\gamma_k,s_n) \to \cf(\gamma_\infty,s_n)$ as $k \to \infty$ for each $n$ fixed, so we can choose $k_n$ sufficiently large that $\cf(\gamma_{k_n},s_n)$ is at most Fréchet distance $\frac1{2n}$ from $\cf(\gamma_\infty,s_n)$ and $t_{k_n}-s_n > n$.

Let $\widetilde \gamma_n = \cf(\gamma_{k_n},s_n)$ and let $\widetilde t_n = t_{k_n}-s_n$.
By construction, $\widetilde \gamma_n$ is at most Fréchet distance $\nicefrac1n$ from $\cf(\gamma_\infty,\infty)$, so $\widetilde \gamma_n \to \cf(\gamma_\infty,\infty)$ as $n \to \infty$.
We just saw that Theorem \ref{theoremFlowConvergesInFrechet} holds for curves converging to a great circle even as time goes to infinity, so we have $\cf(\gamma_{k_n}, t_{k_n}) = \cf(\widetilde \gamma_n, \widetilde t_n) \to \cf(\gamma_\infty,\infty)$, but that contradicts that $\cf(\gamma_k,t_k)$ is bounded away from $\cf(\gamma_\infty,\infty)$.
Thus, $\cf(\gamma_k,t_k) \to \cf(\gamma_\infty,\infty)$ in Fréchet distance.
\end{proof}

\bibliographystyle{plain}
\bibliography{bifl-1}

\end{document}